\documentclass[11pt,a4paper, reqno]{amsart}
\usepackage{amsmath,amssymb}
\usepackage{bm}
\usepackage[dvipdfmx]{graphicx}
\usepackage{float}
\usepackage[T1]{fontenc}
\usepackage{textcomp}
\usepackage{type1cm}
\usepackage{bbm} 
\usepackage{pifont} 
\usepackage{amsthm} 
\usepackage{stmaryrd} 
\usepackage{cite} 
\usepackage{amsthm} 
\usepackage{ascmac} %
\usepackage[all]{xy} 
\usepackage{tikz}
 \usetikzlibrary{arrows} 
\usepackage{mathrsfs} 
\usepackage{url}
\usepackage{alltt} 
\usepackage{braket} 
\usepackage{cases} 

\usepackage{mathtools}
\mathtoolsset{showonlyrefs,showmanualtags}  

\makeatletter
\@addtoreset{equation}{section}

\makeatother

\theoremstyle{plain}
\newtheorem{dfn}[subsection]{Definition}
\newtheorem{thm}[subsection]{Theorem}
\newtheorem{cor}[subsection]{Corollary}

\newtheorem{prp}[subsection]{Proposition}
\newtheorem{lem}[subsection]{Lemma}

\theoremstyle{definition}

\newtheorem{rem}[subsection]{Remark}
\newtheorem*{ac}{Acknowledgements} 

\newcommand{\rarrow}{\longrightarrow}

\newcommand{\powser}[1]{\llbracket #1 \rrbracket} 

\renewcommand{\hat}[1]{\widehat{#1}}

\newcommand{\bb}[1]{\mathbb{#1}}
\newcommand{\fr}[1]{\mathfrak{#1}}
\newcommand{\cal}[1]{\mathcal{#1}}
\newcommand{\scr}[1]{\mathscr{#1}}

\DeclareMathOperator{\Hom}{Hom}
\DeclareMathOperator{\Aut}{Aut}
\DeclareMathOperator{\Tr}{Tr}
\DeclareMathOperator{\Ker}{Ker}

\DeclareMathOperator{\Coker}{Coker}
\DeclareMathOperator{\loc}{loc}
\DeclareMathOperator{\Cor}{Cor}
\DeclareMathOperator{\Res}{Res}
\DeclareMathOperator{\pr}{pr}
\DeclareMathOperator{\corank}{corank}

\DeclareMathOperator{\Gal}{Gal}
\DeclareMathOperator{\Sel}{Sel}

\DeclareMathOperator{\Z}{\mathbb{Z}}
\DeclareMathOperator{\Q}{\mathbb{Q}}

\DeclareMathOperator{\C}{\mathbb{C}}

\DeclareMathOperator{\F}{\mathbb{F}}

\title[]{On non-trivial $\Lambda$-submodules with finite index of the plus/minus Selmer group over anticyclotomic $\Z_{p}$-extension at inert primes}
\author{Ryota Shii}
\address{Faculty Of Mathematics, Kyushu University, Motooka 744, Nishi-ku Fukuoka 819-0395, Japan}
\email{shii.ryota@gmail.com}
\date{}

\begin{document}
\begin{abstract}
  Let $K$ be an imaginary quadratic field where $p$ is inert.
  Let $E$ be an elliptic curve defined over $K$ and suppose that $E$ has good supersingular reduction at $p$.
  In this paper, we prove that the plus/minus Selmer group of $E$ over the anticyclotomic $\Z_{p}$-extension of $K$ has no non-trivial $\Lambda$-submodules of finite index under mild assumptions for $E$.
  This is an analogous result to R. Greenberg and B. D. Kim for the anticyclotomic $\Z_{p}$-extension essentially.
  By applying the results of A. Agboola--B. Howard or A. Burungale--K. B\"uy\"ukboduk--A. Lei, we can also construct examples satisfying the assumptions of our theorem.
\end{abstract}
\maketitle

\section{Introduction}
\noindent
Let $p \geq 5$ be a prime number, $F$ a number field, and $E$ an elliptic curve defined over $F$.
The Pontryagin dual of the $p^{\infty}$-Selmer group of $E$ over a $\Z_{p}$-extension is an important object in Iwasawa theory for elliptic curves. 
It is a fundamental problem in Iwasawa theory whether or not the dual of $p^{\infty}$-Selmer group has no non-trivial finite $\Lambda$-submodule, and this problem has some applications to the $p$-part of Birch and Swinnerton-Dyer (BSD) conjecture.
For example, R. Greenberg \cite{Gre97} proved that if $E$ has good ordinary reduction at all primes of $F$ lying over $p$, and the dual of the $p^{\infty}$-Selmer group over the cyclotomic $\Z_{p}$-extension of $F$ is a $\Lambda$-torsion, then it has no non-trivial finite $\Lambda$-submodules.
In addition, if the $p^{\infty}$-Selmer group over $F$ is a finite group, then he obtains the relation between the constant term of the characteristic polynomial of the $p^{\infty}$-Selmer group and BSD-invariants of $E$ over $F$ by computing the Euler characteristic (in detail, see \cite{Gre99}).
This relation can be seen as a study of the $p$-adic part of the BSD conjecture.
There are many analogous results to an elliptic curve that has supersingular reduction at $p$, for example, by B. D. Kim and T. Kitajima--R. Otsuki.
In detail, see \cite{Kim13}, \cite{Kit-Ots18}.
However, these are results for the case of the cyclotomic $\Z_{p}$-extension, and not many results for the anticyclotomic $\Z_{p}$-extension.
In this paper, we consider an imaginary quadratic field $K$ such that $p$ is \textit{inert} in $K/\Q$ and the anticyclotomic $\Z_{p}$-extension $K_{\infty}$ of $K$.

The anticyclotomic $\Z_{p}$-extension $K_{\infty}$ of $K$ is the unique $\Z_{p}$-extension of $K$ such that the extension $K_{\infty}/\Q$ is Galois and $\Gal(K_{\infty}/\Q)$ is isomorphic to the pro $p$-dihedral group $\Z_{p} \rtimes (\Z/2\Z)$.
This $\Z_{p}$-extension has many arithmetic applications, for instance, Heegner points, Gross--Zagier theorem, etc.
One of the difficulties of the anticyclotomic Iwasawa theory is caused by the fact that the Galois group of $K_{\infty}/\Q$ is non-abelian.
Nonetheless, in the split case, it can usually be reduced to the case of the cyclotomic $\Z_{p}$-extension because $\Gal(K_{\infty, w}/\Q_{p})$ becomes abelian, where $K_{w}$ is the completion fields of $K_{\infty}$ at a prime $w$ lying over $p$.
Indeed, the anticyclotomic Iwasawa theory in the split case is well-known, for example, by \cite{Iov-Pol06}.
On the other hand, the difficulty remains in the inert case since $\Gal(K_{\infty, w}/\Q_{p})$ is still non-abelian.
In \cite{Rub87}, K. Rubin observed and conjectured some Iwasawa-theoretic phenomena for CM elliptic curve by $K$ with a good supersingular reduction at $p$ in this case (these are often called ``Rubin's Conjecture''), but they were not resolved or well-developed until recently.
A recent breakthrough has been made by A. Burungale, S. Kobayashi, and K. Ota \cite{BKO21}, and expects to obtain many analogous results to the split case.
The following theorem, which is the main theorem in this paper, can be proved by using their result.
Note that our theorem does not need the assumption of the complex multiplication but only the formal group of $E$ is isomorphic to a CM elliptic curve over the ring of integers $\cal{O}_{p}$ of the completion field $K_{p}$.

\begin{thm}[Theorem \ref{主定理}]\label{主定理イントロ}
  Let $F$ be a finite extension of $K$ and $h_{K}$ the class number of $K$.
  Assume that $p$ splits completely over $F/K$, and $p \nmid h_{K}[F: K]$.
  Let $E$ be an elliptic curve defined over $F$ that has supersingular reduction at any prime lying over $p$. 
  Suppose that the formal group $\hat{E}$ of $E$ at any primes lying over $p$ is isomorphic to the Lubin-Tate formal group of height 2 with a parameter $-p$ over $\cal{O}_{p}$.
  Then, if $\Sel_{p}^{\pm}(E/F_{\infty})$ is a $\Lambda$-cotorsion, then $\Sel_{p}^{\pm}(E/F_{\infty})$ has no non-trivial $\Lambda$-submodules with finite index.
\end{thm}

\begin{rem}
  \begin{enumerate}
    \item The assumption $p$ splits completely in $F/K$ implies that $F_{\fr{p}} = K_{p}$ for any prime $\fr{p}$ of $F$ lying over $p$. 
    By the assumption $p \nmid h_{K}$, $p$ is totally ramified in $K_{\infty}/K$.
    Thus, any prime of $F$ lying over $p$ is also totally ramified in $F_{\infty}/F$ because $p$ does not divide $[F: K]$.
    \item There exist elliptic curves satisfying the assumption in Theorem \ref{主定理イントロ}.
    For example, if $E$ is defined over $\Q$ and has good supersingular at $p$ and we suppose $p>5$, then the formal group of $E$ is the Lubin--Tate formal group of height 2 with a parameter $-p$.
    If $E$ has a complex multiplication by the integers ring $\cal{O}_{K}$ of $K$, then $E$ also satisfies the assumptions in Theorem \ref{主定理イントロ}.
  \end{enumerate}  
\end{rem}

Applying the results by A. Agboola--B. Howard and A. Burungale--K. B\"uy\"ukboduk--A. Lei to our main theorem, we also obtain the examples as follows.

\begin{cor}[CM case]
  Assume that an elliptic curve $E$ is defined over $\Q$ with good supersingular reduction at $p$, and $F = K$.
  Assume that $E$ has complex multiplication by the ring of integers of $K$.
  Let $\varepsilon$ be the sign in the functional equation of $L(E/\Q, s)$.
  Then, $\Sel_{p}^{-\varepsilon}(E/F_{\infty})$ has no non-trivial $\Lambda$-submodules with finite index.
\end{cor}

\begin{proof}
  By \cite[Theorem 3.6]{Agb-How05}, $\Sel_{p}^{-\varepsilon}(E/F_{\infty})$ is a $\Lambda$-cotorsion under this situation.
  Therefore, this corollary holds because we verify to satisfy the assumptions of Theorem \ref{主定理イントロ}.
\end{proof}

\begin{cor}[non-CM case]
  Let $E$ be an elliptic curve defined over $\Q$ of conductor $N$ with good supersingular reduction at $p$.
  Write $N=N^{+}N^{-}$, where $N^{+}$ and $N^{-}$ are divisible only by primes which are split and inert in $K/\Q$, respectively. 
  Assume that $K$ satisfies $(D_{K}, pN)=1$ and $N^{-}$ is a square-free product of odd number of primes.
  Suppose also that 
  \begin{itemize}
    \item Either $p=5$ and the mod $5$ Galois representation $G_{\Q} \rarrow \Aut_{\F_{5}}(E[5])$ is surjective, or $p > 5$ and the mod $p$ Galois representation $G_{\Q} \rarrow \Aut_{\F_{p}}(E[p])$ is irreducible.
    \item For any prime $\ell \mid N^{-}$ with $\ell^{2} \equiv 1 \bmod p$, the inertia subgroup $I_{\ell} \subset G_{\Q_{\ell}}$ acts non-trivially on $E[p]$.
  \end{itemize}
  Then, $\Sel_{p}^{\pm}(E/K_{\infty})$ has no non-trivial $\Lambda$-submodules with finite index.
\end{cor}
\begin{proof}
  Let $L_{p}(E, K)^{\pm}$ be a $p$-adic $L$-function for $E$ defined in \cite{BBL22}.
  Corollary 3.10 in \cite{BBL22} claims that both $L_{p}(E,K)^{\pm}$ are non-zero.
  Therefore, we see that $\Sel_{p}^{\pm}(E/F_{\infty})$ is a $\Lambda$-cotorsion by \cite[Theorem 1.1]{BBL22}, and we obtain the corollary directly by Theorem \ref{主定理イントロ}.
\end{proof}

As Greenberg's result, we also obtain the analogous result between the characteristic polynomial of the plus/minus Selmer group and BSD-invariants of an elliptic curve.

\begin{cor}
  Under the situation of Theorem \ref{主定理イントロ}, suppose $\Sel_{p}(E/F)$ is, in addition, a finite group.
  Then, $\Sel_{p}^{\pm}(E/F_{\infty})$ is a $\Lambda$-cotorsion and if let $(f^{\pm}) \subset \Lambda$ denote the characteristic ideals of $\Sel_{p}^{\pm}(E/F_{\infty})^{\vee}$, then we have
  \begin{align}
    |f^{\pm}(0)| \sim \# \Sel_{p}(E/F) \cdot \prod_{v} c_{v},
  \end{align}
  where $c_{v}$ is the Tamagawa number at a prime $v$ of $F$.
\end{cor}

\begin{ac}
  The author is thankful to Shinichi Kobayashi for reading the manuscript carefully and pointing out mathematical mistakes.
  He would like to thank Ashay Burungale, Kazuto Ota and Takenori Kataoka for the discussions and for telling him some references.
  He is grateful to Keiichiro Nomoto, Satoshi Kumabe, Taiga Adachi, and Akihiro Goto for supporting him. 
  The author was supported by JST SPRING, Grant Number JPMJSP2136.
\end{ac}

\section{Rubin's conjecture}
\noindent
In this section, we recall Rubin's conjecture, which is a property of a Lubin--Tate formal group of height 2, and prepare some propositions used after sections.

Fix a prime number $p \geq 5$.
Let $\Phi$ be the unramified quadratic extension of $\Q_{p}$ and $\cal{O}$ its ring of integers.
We fix a Lubin--Tate formal group $\scr{F}$ over $\cal{O}$ for the uniformizing parameter $\pi:= -p$.
For $n \geq -1$, we denote $\Phi_{n}=\Phi(\mathscr{F}[\pi^{n+1}])$ by the extension of $\Phi$ in $\C_{p}$ generated by the $\pi^{n+1}$-torsion points of $\scr{F}$, and put $\Phi_{\infty}=\bigcup_{n \geq 0} \Phi_{n}$.
Then, the Galois action on the $\pi$-adic Tate module $T_{\scr{F}}:=T_{\pi}\scr{F}$ of $\scr{F}$ induces a natural isomorphism
\begin{align}
 \kappa_{\scr{F}}:\Gal(\Phi_{\infty}/\Phi) \rarrow \Aut(T_{\pi} \scr{F}) \simeq \cal{O}^{\times} \simeq \Delta \times \cal{O}.
\end{align}
by the Lubin--Tate theory, where $\Delta=\Gal(\Phi_{0}/\Phi) \simeq (\cal{O}/\pi\cal{O})^{\times}$.
By the natural action of $\Gal(\Phi/\Q_{p})$ on $G:=\Gal(\Phi_{\infty}/\Phi_{0})$,
we have a canonical decomposition $G \simeq G^{+} \oplus G^{-}$, $G^{\pm} \simeq \Z_{p}$,
where $G^{+}$ (resp. $G^{-}$) is the maximal subgroup of $G$ on which $\Gal(\Phi/\Q_{p})$ acts via the trivial (resp. non-trivial) character.
The group $G^{-}$ is the Galois group of the anticyclotomic $\Z_{p}$-extension $\Psi_{\infty}$ of $\Phi$.
We fix a topological generator $\gamma$ of $G^{-}$.

For any $n$ let $U_{n}$ be the group of principle units in $\Phi_{n}$, in other words, the group of units in the integer ring of $\Phi_{n}$ that are congruent to 1 modulo the maximal ideal.
We put $T_\scr{F}^{\otimes -1}=\Hom_{\cal{O}}(T_\scr{F}, \cal{O})$ and
\begin{align}
    U_{\infty}^{*}:=\left( \varprojlim_{n}(U_{n} \otimes_{\Z_{p}} T_\scr{F}^{\otimes -1}) \right)^{\Delta},
\end{align}
where the inverse limit is taken to the norm maps, and for $\cal{O}[\Delta]$-module $A$, $A^{\Delta}$ denotes the submodule of $A$ fixed by $\Delta$.
Define the Iwasawa algebras $\Lambda_{\mathcal{O}, 2}=\mathcal{O}\llbracket G \rrbracket$, $\Lambda_{\mathcal{O}}:=\mathcal{O} \llbracket G^{-} \rrbracket$.
It is known that $U_{\infty}^{*}$ is a free $\Lambda_{\cal{O}, 2}$-module of rank 2.
We put $V_{\infty}^{*}:=U_{\infty}^{*} \otimes_{\Lambda_{\cal{O}, 2}} \Lambda_{\cal{O}}$.

For any $u^{*}=u \otimes v \in U^{*}$ ($u=( u_{n})_{n} \in U_{\infty}$, $v=(v_{n})_{n} \in T_{\scr{F}}^{\otimes -1}$), there exists the unique power series $f_{u^{*}} \in \cal{O}\powser{T}^{\times}$ such that $f_{u^{*}}(v_{n})=u_{n}$ for every $n$, which is called the Coleman power series.
Using this, we define the Coates--Wiles logarithmic derivatives $\delta_{n}:U_{\infty}^{*} \rarrow \Phi_{n}$ by
\begin{align}
 \delta_{n}(u^{*})=\frac{1}{\lambda'(v_{n})} \frac{f'_{u^{*}}(v_{n})}{f_{u^{*}}(v_{n})},
\end{align}
where $\lambda$ is the formal logarithm of $\scr{F}$.
For a finite character $\chi:\Gal(\Phi_{\infty}/\Phi) \rarrow \overline{\Q_{p}}^{\times}$ of conductor dividing $p^{n+1}$ and $u^{*} \in U_{\infty}^{*}$, we put
\begin{align}
 \delta_{\chi}(u^{*})=\frac{1}{\pi^{n+1}} \sum_{\sigma \in \Gal(\Phi_{n}/\Phi)} \chi(\sigma) \delta_{n}(u^{*})^{\sigma}.
\end{align}
It is known that the definition is independent of the choice of $n$.
If $\chi$ factors through $G^{-}$, then $\delta_{\chi}$ factors through $V_{\infty}^{*}$.

Let $\Xi$ be the set of finite characters of $G^{-}$,
\begin{align}
 \Xi^{+} &=\{ \chi \in \Xi ~\vline~ \textrm{the conductor of $\chi$ is an even power of $p$} \}, \\
 \Xi^{-} &=\{ \chi \in \Xi ~\vline~ \textrm{the conductor of $\chi$ is an odd power of $p$} \}, \\
 \textrm{ and } 
 V_{\infty}^{*, \pm} &=\{ v \in V_{\infty}^{*} ~\vline~ \delta_{\chi}(v)=0 \textrm{ for all } \chi \in \Xi^{\mp} \}.
\end{align}

Rubin show that $V_{\infty}^{*, \pm}$ is free $\Lambda_{\cal{O}}$-module of rank 1 (\textit{cf.} \cite{Rub87}).
The following theorem is conjectured by K. Rubin in \cite{Rub87}, and proved by A. Burungale, S. Kobayashi, and K. Ota (\textit{cf.} \cite{BKO21}).

\begin{thm}\label{Rubin予想}
 $V_{\infty}^{*}=V_{\infty}^{*, +} \oplus V_{\infty}^{*, -}$.
\end{thm}

For any $n \geq 0$, put $\Lambda_{\mathcal{O}, n}=\mathcal{O}[\Gal(\Psi_{n}/\Phi)]$ and let $\Xi_{n}^{\pm}$ be the set of $\chi \in \Xi^{\pm}$ factoring through $\Gal(\Psi_{n}/\Phi)$.
For $\chi \in \Xi_{n}^{\pm}$, we define
\begin{align}
 \lambda_{\chi}(x) = \frac{1}{p^{n}} \sum_{\sigma \in \Gal(\Psi_{n}/\Phi)} \chi^{-1}(\sigma)\lambda(x)^{\sigma}.
\end{align}
Our main objects in this section are the $\Lambda_{\cal{O}, n}$-submodules
\begin{align}
 \scr{F}(\fr{m}_{n})^{\pm}=\{ x \in \scr{F}(\fr{m}_{n}) ~\vline~ \lambda_{\chi}(x)=0 \textrm{ for all } \chi \in \Xi_{n}^{\pm} \}
\end{align}
of $\scr{F}(\fr{m}_{n})$, where $\fr{m}_{n}$ is the maximal ideal of $\Psi_{n}$.
We write $\varphi_{p^{k}}$ for the $p^{k}$-th cyclotomic polynomial for a positive integer $k$.
We put
\begin{align}
  \omega_{n}^{+}=\omega_{n}^{+}(\gamma)=\prod_{\substack{1 \leq k \leq n \\ k: \textrm{even}}} \varphi_{p^{k}}(\gamma), \quad
  \omega_{n}^{-}=\omega_{n}^{-}(\gamma)=(\gamma -1)\prod_{\substack{1 \leq k \leq n \\ k:\mathrm{odd}}} \varphi_{p^{k}}(\gamma)
\end{align}
We also put $\omega_{0}^{+}:=1$, $\omega_{0}^{-}:=\gamma -1$.
Let $v_{\pm}$ denote a basis of $V_{\infty}^{*, \pm}$, $v_{\pm, n}$ the image of $v_{\pm}$ in the quotient $V_{\infty}^{*}/(\gamma^{p^{n}}-1)$ and
\begin{align}
  c_{n}^{\pm}:=\omega_{n}^{\mp}v_{\pm, n} \in V_{\infty}^{*}/(\gamma^{p^{n}}-1).
\end{align}
By Theorem \ref{Rubin予想}, the $\Lambda_{\cal{O},n}$-structure of $\scr{F}(\fr{m}_{n})^{\pm}$ turns out as follows:

\begin{thm}[{\cite[Lamma 2.3, Theorem 2.4]{BKOpre}}]\label{プラマイ分解}
\begin{enumerate}
  \item Let $n \geq 1$.
  If $(-1)^{n+1}=\pm1$, then we have
  \begin{align}
    \Tr_{n+1/n}c_{n+1}^{\pm}=c_{n-1}^{\pm}, \quad c_{n}^{\pm}=c_{n-1}^{\pm}.
  \end{align}
  Here, $\Tr_{n+1/n}:\scr{F}(\fr{m}_{n+1}) \rarrow \scr{F}(\fr{m}_{n})$ is the trace map.

  \item $c_{n}^{\pm} \in \mathscr{F}(\mathfrak{m}_{n})^{\pm}$,  and $\mathscr{F}(\mathfrak{m}_{n})^{\pm}$ are generated by $c_{n}^{\pm}$ as $\Lambda_{\mathcal{O}, n}$-module.
  In particular, it holds that
  \begin{align}
    \Hom_{\mathcal{O}}(\mathscr{F}(\mathfrak{m}_{n})^{\pm} \otimes \Phi/\mathcal{O}, \Phi/\mathcal{O}) \simeq \Lambda_{\mathcal{O}, n}.
  \end{align}

  \item $\mathscr{F}(\mathfrak{m}_{n})=\mathscr{F}(\mathfrak{m}_{n})^{+} \oplus \mathscr{F}(\mathfrak{m}_{n})^{-}$ as $\Lambda_{\mathcal{O}, n}$-module.
\end{enumerate}
\end{thm}

For $f \in \Lambda_{\cal{O}, n}$, let $f^{\iota}$ denote the image of $f$ under the homomorphism of $\cal{O}$-algebra given by $\gamma \mapsto \gamma^{-1}$.

\begin{prp}\label{形式群が消える}
  We have $\omega_{n}^{\pm}\mathscr{F}(\mathfrak{m}_{n})^{\pm}=0$ and $(\omega_{n}^{\pm})^{\iota}\mathscr{F}(\mathfrak{m}_{n})^{\pm}=0$.
\end{prp}
\begin{proof}
  By Theorem \ref{プラマイ分解}(2), it suffices to show that $\omega_{n}^{\pm}c_{n}^{\pm}=0$.
  Since we have $\omega_{n}^{\pm}\omega_{n}^{\mp}=\gamma^{p^{n}}-1$, we see that
  \begin{align}
    \omega_{n}^{\pm}c_{n}^{\pm}=\omega_{n}^{\pm}\omega_{n}^{\mp}v_{\pm, n}=(\gamma^{p^{n}}-1)v_{\pm, n}=0
  \end{align}
  in $V_{\infty}^{*}/(\gamma^{p^{n}}-1)$.
  The latter equation can be proved similarly.
\end{proof}

\begin{prp}\label{Key prop}
  Let $\Psi_{\infty}$ be the maximal ideal of $\Psi_{\infty}$.
  Then, we have
  \begin{align}
    \Hom_{\mathcal{O}}((\mathscr{F}(\mathfrak{m}_{\infty})^{\pm} \otimes_{\mathcal{O}} \Phi/\mathcal{O}), \Phi/\mathcal{O}) \simeq \Lambda_{\mathcal{O}}.
  \end{align}
\end{prp}
\begin{proof}
  By Corollary 5.7 in \cite{Rub87}, the pairing $\langle ~,~ \rangle:\mathscr{F}(\mathfrak{m}_{\infty})^{\pm} \otimes_{\mathcal{O}} \Phi/\mathcal{O} \times V_{\infty}^{*}/V^{\pm, *} \rarrow \hat{E}[p^{\infty}] \simeq \Phi/\mathcal{O}$ is non-degenerate.
  By Theorem \ref{Rubin予想}, we have $V_{\infty}^{*}/V^{\pm, *} \simeq V^{\mp, *} \simeq \Lambda_{\mathcal{O}}$.
\end{proof}

\section{Global settings and Local conditions}
\noindent
It retains the symbols in the previous section.
Let $K$ be an imaginary quadratic field such that $p$ is inert in $K/\Q$ and suppose that the class number $h_{K}$ of $K$ is not divided by $p$.
Let $F$ denote an algebraic number field over $K$.
We assume that the prime $p$ of $K$ is split completely in $F$, and $p \nmid [F: K]$.
For any prime ideals $\fr{p}$ of $F$ lying over $p$, let $\cal{O}_{\fr{p}}$ be the ring of integers of the completion field $F_{\fr{p}}$ of $F$ at $\fr{p}$.
Let $K_{\infty}/K$ be the anticyclotomic $\Z_{p}$-extension and $F_{\infty}$ the composition field of $F$ and $K_{\infty}$.
By the assumption $p \nmid [F:K]$, we have a canonical isomorphism $\Gal(F_{\infty}/F) \simeq \Gal(K_{\infty}/K)$.
Let $F_{n}$ denote the subfield of $F_{\infty}$ such that $[F_{n}:F]=p^{n}$.

Let $E$ be an elliptic curve defined over $F$.
We assume that $E$ has a good supersingular reduction at any prime $\fr{p}$ of $F$ lying over $p$ and the formal group $\hat{E}$ of $E$ defined over $\cal{O}_{\fr{p}}$ is isomorphic to the Lubin-Tate formal group $\scr{F}$.
Note that $\cal{O}_{\fr{p}}$ is seen as the ring of integers $\cal{O}_{p}$ of $K_{p}$ because $p$ is split completely in $F/K$.
Let $T=T_{p}E$ be the $p$-adic Tate module, $V:=T \otimes \Q_{p}$ and $A:=V/T$.

Let $H$ be the Hilbert class field of $K$, $H_{n}:=H(E[p^{n+1}])$ for $n \geq 0$, and $H_{\infty}=\bigcup_{n} H_{n}$.
Then, $p$ is split completely in $H/K$.
Fix a prime $\fr{P}$ of $H$ lying over $p$, that is embedding $H \hookrightarrow H_\fr{P}$ is determined uniquely.
Since we have $\hat{E} \simeq \scr{F}$ over $\cal{O}_{p}$, $\fr{P}$ is totally ramified in $H_{\infty}/H$ and the completion of $H_{n}$ at $\fr{P}$ is 
\begin{align}
  \Phi(E[p^{n+1}])=\Phi(\hat{E}[p^{n+1}])=\Phi(\mathscr{F}[p^{n+1}])=\Phi_{n}.
\end{align}
Therefore, we see that $\Gal(H_{\infty}/H) \simeq \Gal(\Phi_{\infty}/\Phi)$ and identify $\Gamma:=\Gal(K_{\infty}/K)$ with $G^{-}$ by $p \nmid h_{K}$.
Let $\Lambda:=\Z_{p}\powser{\Gamma}$ be the Iwasawa algebra.
Note that $\Lambda_{\cal{O}_{p}}$ is a free $\Lambda$-module of rank 2,
and for $\Lambda_{\cal{O}_{p}}$-module $M$, $\Hom_{\cal{O}_{p}}(M, K_{p}/\cal{O}_{p})$ can be identified with $M^{\vee}$ as a $\Z_{p}$-module.

Let $\Sigma$ denote a finite set of places of $F$ containing all primes above $p$, bad primes of $E$, and infinity places.
The maximal extension of $F$ unramified outside $\Sigma$ is denoted by $F_{\Sigma}$.

\begin{dfn}[plus/minus Selmer group]
  For any places $w$ of $F_{\infty}$ lying over $p$, put $\mathbb{H}_{w}^{\pm}:=\hat{E}(F_{\infty, w})^{\pm} \otimes_{\mathcal{O}_{\fr{p}}} F_{\fr{p}}/\mathcal{O}_{\fr{p}} \subset H^{1}(F_{\infty, w}, A)$.
  We define the plus/minus Selmer group of $E$ over $F_{\infty}$ by
  \begin{align}
    \Sel_{p}^{\pm}(E/F_{\infty}):=\Ker \left( H^{1}(F_{\Sigma}/F_{\infty}, A) \rarrow \prod_{v_{\infty} \mid \fr{q},~ \fr{q} \in \Sigma \setminus \{ \fr{q} \}} H^{1}(F_{\infty, v_{\infty}}, A) \times \prod_{w \mid \fr{p}, ~\fr{p} \mid p} \frac{H^{1}(F_{\infty, w}, A)}{\bb{H}_{w}^{\pm}} \right).
  \end{align}
\end{dfn}

In this section, fix a prime $\fr{p}$ of $F$ lying over $p$ and the prime $w$ of $F_{\infty}$ lying over $\fr{p}$, and we observe the structure of $\mathbb{H}_{w}^{\pm}$ as $\Lambda$-module.
Since $\fr{p}$ is totally ramified in $F_{n}/F$, we denote the prime of $F_{n}$ lying over $\fr{p}$ as also $\fr{p}$.
In order to use a ``trick of twisting'', fix an isomorphism $\kappa:G^{-} \rarrow 1+p\Z_{p}$, and for any $s \in \Z_{p}$ define $T_{s}:=T \otimes \kappa^{s}$ as a twist of $T$ by $s$.
Similarly, $A_{s}$ is also defined.

\begin{prp}\label{p-torはKinfの外}
  For any positive integer $n$, $E(F_{n, \fr{p}})$ has no $p$-power torsion.
  In particular, $E(F_{\infty, w})$ also has no $p$-power torsion.
\end{prp}
\begin{proof}
  Since $E$ has good supersingular reduction at $p$, we see that $E[p] \simeq \hat{E}[p]$.
  Therefore, it is sufficient to show that $\hat{E}(\fr{m}_{n})$ has no $p$-power torsion points, where $\fr{m}_{n}$ is the maximal ideal of $F_{n, \fr{p}}$.
  We assume that $\hat{E}(\fr{m}_{n})$ has $p$-power torsion points.
  Then, we see that $F_{\fr{p}}(\hat{E}[p]) \subset F_{n,\fr{p}}$.
  We have $[F_{n, \fr{p}}:F_{\fr{p}}]=p^{n}$, on the other hand, $[F_{\fr{p}}(\hat{E}[p]):F_{\fr{p}}]=p^{2}-1$ by the Lubin-Tate theory.
  It is a contradiction.
\end{proof}

Proposition \ref{p-torはKinfの外} implies $A^{G_{F_{n,\fr{p}}}}=A^{G_{F_{\infty, w}}}=0$ for any positive integer $n$, and $A_{s}^{G_{F_{n, \fr{p}}}}=A_{s}^{G_{F_{\infty, w}}}=0$ for any $s \in \Z_{p}$.
Hence, we have an isomorphism
\begin{align}
  H^{1}(F_{n,\fr{p}}, A_{s}) \overset{\sim}{\rarrow} H^{1}(F_{\infty, w}, A_{s})^{\Gal(F_{\infty, w}/F_{n,\fr{p}})} \label{同一視1}
\end{align}
by the inflation--restriction exact sequence.
We identify $H^{1}(F_{n,\fr{p}}, A_{s})$ with $H^{1}(F_{\infty, w}, A_{s})^{\Gal(F_{\infty, w}/F_{n,\fr{p}})}$ by the isomorphism \eqref{同一視1}.

Similarly, for any positive integer $k$, the long cohomology exact sequence
\begin{align}
 A_{s}^{G_{F_{n, \fr{p}}}} \rarrow H^{1}(F_{n, \fr{p}}, A_{s}[p^{k}]) \rarrow H^{1}(F_{n, \fr{p}}, A_{s})[p^{k}] \rarrow 0
\end{align}
induces the isomorphism 
\begin{align}
 H^{1}(F_{n, \fr{p}}, A[p^{k}]) \overset{\sim}{\rarrow} H^{1}(F_{n, \fr{p}}, A)[p^{k}]
\end{align}
by Proposition \ref{p-torはKinfの外}.
Here, $H^{1}(F_{n, \fr{p}}, A_{s})[p^{k}]$ is the subgroup of $H^{1}(F_{n, \fr{p}}, A_{s})$ consisting of $p^{k}$-torsion.

\begin{dfn}
  We define
  \begin{align}
    \mathbb{H}_{n, \fr{p}}^{\pm}&:=(\mathbb{H}_{w}^{\pm})^{\Gal(F_{\infty, w}/F_{n, \fr{p}})} \subset H^{1}(F_{n, \fr{p}}, A).
  \end{align}
  For any $s \in \Z_{p}$, we define
  \begin{align}
    \mathbb{H}_{n, \fr{p}}^{s, \pm}:=(\mathbb{H}_{w}^{\pm} \otimes \kappa^{s})^{\Gal(F_{\infty, w}/F_{n, \fr{p}})} \subset H^{1}(F_{n, \fr{p}}, A_{s}).
  \end{align}
\end{dfn}

Since we have $(\mathbb{H}_{w}^{\pm})^{\vee} \simeq \Lambda_{\cal{O}_{\fr{p}}} \simeq \Lambda^{2}$ as $\Lambda$-module by Proposition \ref{Key prop}, we see that $(\mathbb{H}_{n, \fr{p}}^{\pm})^{\vee} \simeq \Lambda_{n}^{2}$ as $\Z_{p}$-module.
Let $M_{n}^{\pm} \subset H^{1}(F_{n, \fr{p}}, T)$ be the annihilator of $\mathbb{H}_{n, \fr{p}}^{\pm}$ with respect to the Tate pairing 
\begin{align}
  H^{1}(F_{n, \fr{p}}, A) \times H^{1}(F_{n,\fr{p}}, T) \rarrow \Q_{p}/\Z_{p}.
\end{align}
Since $\mathbb{H}_{n , \fr{p}}^{\pm}$ is $p$-divisible, $H^{1}(F_{n, \fr{p}}, T)/M_{n}^{\pm}$ has no $p$-power torsion elements.
Hence, for any positive integer $j$, $M_{n}^{\pm}/p^{j}M_{n}^{\pm}$ is the exact annihilator of $\mathbb{H}_{n, \fr{p}}^{\pm}[p^{j}]$ with respect to the pairing $\langle ~, ~\rangle_{n}:H^{1}(F_{n, \fr{p}}, A[p^{j}]) \times H^{1}(F_{n, \fr{p}}, A[p^{j}]) \rarrow \Z/p^{j}\Z$.

\begin{prp}\label{twistなしのexact annihilator}
  For any positive integers $j$ and $n$, we have $M_{n}^{\pm}/p^{j}M_{n}^{\pm}=\mathbb{H}_{n, \fr{p}}^{\pm}[p^{j}]$.
  In other words, the exact annihilator of $\mathbb{H}_{n, \fr{p}}^{\pm}[p^{j}]$ is itself.
\end{prp}
The proof of Proposition \ref{twistなしのexact annihilator} is the same way as Proposition 3.15 in \cite{Kim07}.
We prove only the case of $-$ because both the cases of $+$ and $-$ are the same.
In order to prove this proposition, we ready some symbols and a lemma.
For $n \leq m$, let $\Res_{n}^{m}:H^{1}(F_{n, \fr{p}}, A_{s}) \rarrow H^{1}(F_{m, \fr{p}}, A_{s})$ be the restriction map and $\Cor_{n}^{m}:H^{1}(F_{m, \fr{p}}, A_{s}) \rarrow H^{1}(F_{n, \fr{p}}, A_{s})$ the corestriction map.

\begin{lem}\label{lem:cor}
  We have $\Cor_{n}^{m}(\mathbb{H}_{m, \fr{p}}^{-}[p^{j}])=\mathbb{H}_{n, \fr{p}}^{-}[p^{j}]$.
\end{lem}
\begin{proof}
  By the isomorphism $(\mathbb{H}_{m, \fr{p}}^{-})^{\vee} \simeq \Lambda_{m}^{2}$ as $\Z_{p}$-module, 
  $\mathbb{H}_{m, \fr{p}}^{-}[p^{j}]$ can be identified with $\Hom_{\Z_{p}}(\Lambda_{m}^{2}, \Z/p^{j}\Z)$.
  The restriction of $\Res_{n}^{m}$ to $\mathbb{H}_{n, \fr{p}}^{-}[p^{j}]$ can be also identified with the injection 
  \begin{align}
    \pr^{*}:\Hom(\Lambda_{n}^{2}, \Z/p^{j}\Z) \rarrow \Hom(\Lambda_{m}^{2}, \Z/p^{j}\Z)
  \end{align}
  from the canonical projection $\pr:\Lambda_{m}^{2} \rarrow \Lambda_{n}^{2}$.
  Since we have $\Res_{n}^{m} \circ \Cor_{n}^{m}=\Tr_{m/n}$, the corestriction map $\Cor_{n}^{m}$ can be seen as the following map $h^{*}$:
  For $x \in \Lambda_{n}^{2}$, take $x' \in \pr^{-1}(x)$ and we define a map $h:\Lambda_{n}^{2} \rarrow \Lambda_{m}^{2}$ as 
  \begin{align}
    h(x):=\prod_{n+1 \leq n' \leq m} \varphi_{p^{n'}}(\gamma)x'.
  \end{align}
  Since we have $x' \in (\gamma^{p^{n}}-1)\Lambda^{2}$ if $\pr(x')=0$, there exists $y' \in \Lambda^{2}$ such that $x'=(\gamma^{p^{n}}-1)y'$.
  Therefore, we see that
  \begin{align}
    \prod_{n+1 \leq n' \leq m} \varphi_{p^{n'}}(\gamma)x'=\prod_{1 \leq n' \leq m} \varphi_{p^{n'}}(\gamma)y'=(\gamma^{p^{m}}-1)y'=0
  \end{align}
  in $\Lambda_{m}^{2}$.
  This follows that the definition of $h$ is independent of the choice of $x'$.
  By $n \leq m$, $h$ is injective.
  Consider the map $h^{*}:\Hom(\Lambda_{m}^{2}, \Z/p^{j}\Z) \rarrow \Hom(\Lambda_{n}^{2}, \Z/p^{j}\Z)$ induced by $h$.
  Since we have $\pr^{*} \circ h^{*}=\Tr_{m/n}$ and $\pr^{*}=\Res^{m}_{n}$, we see that $h^{*}=\Cor_{n}^{m}$.
  Identifying these maps, it is sufficient to prove that $h^{*}$ is surjective.
  Since $\Coker h \simeq \left( \Z_{p}[X]/\left( \prod_{n+1 \leq n' \leq m} \varphi_{n'}(X+1) \right) \right)^{2}$ is free $\Z_{p}$-module, we see that
  \begin{align}
    0 \rarrow \Lambda_{n}^{2} \overset{h}{\rarrow} \Lambda_{m}^{2} \rarrow \Coker h \rarrow 0
  \end{align}
  is exact.
  Therefore, $h^{*}$ is surjective.
\end{proof}

\begin{rem}\label{rem:cor}
  We can similarly prove $\Cor_{n}^{m}(\mathbb{H}_{m, \fr{p}}^{s, -}[p^{j}])=\mathbb{H}_{n, \fr{p}}^{s, -}[p^{j}]$ for any $s \in \Z_{p}$ and $n \leq m$.
\end{rem}

\begin{proof}[Proof of Proposition \ref{twistなしのexact annihilator}]
  First, we prove $\mathbb{H}_{n, \fr{p}}^{-}=\mathbb{H}_{n, \fr{p}}^{-}[\omega_{n}^{-}]+\mathbb{H}_{n, \fr{p}}^{-}[\omega_{n}^{+}]$.
  Since we have $\mathbb{H}^{-}_{n, \fr{p}} \simeq (\Lambda_{n}^{2})^{\vee}$, we have the isomorphisms 
  \begin{align}
    \mathbb{H}_{n, \fr{p}}^{-}[\omega_{n}^{-}] \simeq \left((\Lambda_{n}/(\omega_{n}^{-})^{\iota})^{2}\right)^{\vee}, \quad
    \mathbb{H}_{n, \fr{p}}^{-}[\omega_{n}^{+}] \simeq \left((\Lambda_{n}/(\omega_{n}^{+})^{\iota})^{2}\right)^{\vee}. \label{プラマイのdual}
  \end{align}

  Therefore, we see that $\corank_{\Z_{p}} \mathbb{H}_{n, \fr{p}}^{-}[\omega_{n}^{-}]=2\deg(\omega_{n}^{-})$, $\corank_{\Z_{p}} \mathbb{H}_{n, \fr{p}}^{-}[\omega_{n}^{+}]=2\deg(\omega_{n}^{+})$.
  Here, it turns out 
  \begin{align}
    \mathbb{H}_{n, \fr{p}}^{-}[\omega_{n}^{-}] \cap \mathbb{H}_{n, \fr{p}}^{-}[\omega_{n}^{+}] \simeq \left( (\Lambda_{n}/((\omega_{n}^{-})^{\iota}+(\omega_{n}^{+})^{\iota}))^{2} \right)^{\vee}. \label{order}
  \end{align}
  Since $\omega_{n}^{-}$ is prime to $\omega_{n}^{+}$, the right hand sides of \eqref{order} is finite.
  Both $\mathbb{H}_{n, \fr{p}}^{-}[\omega_{n}^{-}]$ and $\mathbb{H}_{n, \fr{p}}^{-}[\omega_{n}^{+}]$ are $p$-divisible, we obtain $\mathbb{H}_{n, \fr{p}}^{-}=\mathbb{H}_{n, \fr{p}}^{-}[\omega_{n}^{-}]+\mathbb{H}_{n, \fr{p}}^{-}[\omega_{n}^{+}]$.

  Next, we prove $\hat{E}(\mathfrak{m}_{n})^{-} \subset M_{n}^{-}$ for the maximal ideal $\fr{m}_{n}$ of $F_{n, \fr{p}}$.
  Since we have $\omega_{n}^{-}\hat{E}(\mathfrak{m}_{n})^{-}=0$ by Proposition \ref{形式群が消える}, we see that $\hat{E}(\mathfrak{m}_{n})^{-} \otimes F_{\fr{p}}/\mathcal{O}_{\fr{p}} \subset \mathbb{H}_{n, \fr{p}}^{-}[\omega_{n}^{-}]$.
  By $(\hat{E}(\mathfrak{m}_{n})^{-} \otimes_{\mathcal{O}_{\fr{p}}} F_{\fr{p}}/\mathcal{O}_{\fr{p}})^{\vee} \simeq \Lambda_{n}^{2}$ as $\Lambda_{n}$-module from Proposition \ref{プラマイ分解}(2), we see that 
  \begin{align}
    \corank_{\Z_{p}} (\hat{E}(\mathfrak{m}_{n})^{-} \otimes_{\mathcal{O}_{\fr{p}}} F_{\fr{p}}/\mathcal{O}_{\fr{p}})
    =2\deg(\omega_{n}^{-})
    =\corank_{\Z_{p}} \mathbb{H}_{n, \fr{p}}^{-}[\omega_{n}^{-}]
  \end{align}
  Since both $\hat{E}(\mathfrak{m}_{n})^{-} \otimes_{\mathcal{O}_{\fr{p}}} F_{\fr{p}}/\mathcal{O}_{\fr{p}}$ and $\mathbb{H}_{n, \fr{p}}^{-}[\omega_{n}^{-}]$ are $p$-divisible, we see that
  \begin{align}
    \hat{E}(\mathfrak{m}_{n})^{-} \otimes_{\mathcal{O}_{\fr{p}}} F_{\fr{p}}/\mathcal{O}_{\fr{p}}=\mathbb{H}_{n, \fr{p}}^{-}[\omega_{n}^{-}]. \label{n-layer}
  \end{align}  
  Therefore, we obtain 
  \begin{align}
    \langle \mathbb{H}_{n, \fr{p}}^{-}[\omega_{n}^{-}], \hat{E}^{-}(\mathfrak{m}_{n}) \rangle_{n}
    =\langle \hat{E}(\mathfrak{m}_{n})^{-} \otimes_{\mathcal{O}_{\fr{p}}} F_{\fr{p}}/\mathcal{O}_{\fr{p}}, \hat{E}(\mathfrak{m}_{n})^{-} \rangle_{n}
    =0.
  \end{align}
  Since we have $\omega_{n}^{-}\mathbb{H}_{n, \fr{p}}^{-} \subset \mathbb{H}_{n, \fr{p}}^{-}[\omega_{n}^{+}]$ and $\Ker(\mathbb{H}_{n, \fr{p}}^{-} \overset{\omega_{n}^{-}}{\rarrow} \mathbb{H}_{n, \fr{p}}^{-})
  =\mathbb{H}_{n,\fr{p}}^{-}[\omega_{n}^{-}]$, we see that $\omega_{n}^{-}\mathbb{H}_{n, \fr{p}}^{-}
  \simeq \mathbb{H}_{n, \fr{p}}^{-}/\mathbb{H}_{n, \fr{p}}^{-}[\omega_{n}^{-}]$.
  By \eqref{プラマイのdual}, we calculate
  \begin{align}
    \corank_{\Z_{p}} \omega_{n}^{-}\mathbb{H}_{n, \fr{p}}^{-}
    =2(p^{n}-\deg(\omega_{n}^{-}))
    =\corank_{\Z_{p}} \mathbb{H}_{n, \fr{p}}^{-}[\omega_{n}^{+}].
  \end{align}
  Since both $\omega_{n}^{-}\mathbb{H}_{n, \fr{p}}^{-}$ and $\mathbb{H}_{n, \fr{p}}^{-}[\omega_{n}^{+}]$ are $p$-divisible, we see that $\omega_{n}^{-}\mathbb{H}_{n, \fr{p}}^{-}=\mathbb{H}_{n, \fr{p}}^{-}[\omega_{n}^{+}]$.
  By $(\omega_{n}^{-})^{\iota}\hat{E}(\mathfrak{m}_{n})^{-}=0$, we see that 
  \begin{align}
    \langle \mathbb{H}_{n, \fr{p}}^{-}[\omega_{n}^{+}], \hat{E}(\mathfrak{m}_{n})^{-} \rangle_{n}
    &= \langle\omega_{n}^{-}\mathbb{H}_{n, \fr{p}}^{-}, \hat{E}(\mathfrak{m}_{n})^{-} \rangle_{n} \\
    &=\langle \mathbb{H}_{n, \fr{p}}^{-}, (\omega_{n}^{-})^{\iota}\hat{E}(\mathfrak{m}_{n})^{-} \rangle_{n} \\
    &=\langle \mathbb{H}_{n, \fr{p}}^{-}, 0 \rangle_{n} \\
    &=0.
  \end{align}
  These show that $\hat{E}(\mathfrak{m}_{n})^{-}$ is the annihilator of $\mathbb{H}_{n, \fr{p}}^{-}=\mathbb{H}_{n, \fr{p}}^{-}[\omega_{n}^{-}]+\mathbb{H}_{n, \fr{p}}^{-}[\omega_{n}^{+}]$, that is $\hat{E}(\mathfrak{m}_{n})^{-} \subset M_{n}^{-}$.
  In particular, we obtain $\hat{E}(\mathfrak{m}_{n})^{-}/p^{j}\hat{E}(\mathfrak{m}_{n})^{-} \subset M_{n}^{-}/p^{j}M_{n}^{-}$.

  Next, we prove that $\mathbb{H}_{n, \fr{p}}^{-}[p^{j}] \subset M_{n}^{-}/p^{j}M_{n}^{-}$ for any $j$.
  To show this inclusion, it is sufficient to show that 
  \begin{align}
    (\hat{E}(\mathfrak{m}_{m})^{-}/p^{j}\hat{E}^{-}(\mathfrak{m}_{m}))^{\Gal(F_{m, \fr{p}}/F_{n, \fr{p}})}
    \subset M_{n}^{-}/p^{j}M_{n}^{-}
  \end{align}
  for any $m \geq n$ because we have $\mathbb{H}_{n, \fr{p}}^{-}[p^{j}]
  =\bigcup_{m=n}^{\infty} \left( \hat{E}(\mathfrak{m}_{m})^{-}/p^{j}\hat{E}(\mathfrak{m}_{m})^{-} \right)^{\Gal(F_{\infty, w}/F_{n, \fr{p}})}$.
  Take any $y \in \left( \hat{E}(\mathfrak{m}_{m})^{-}/p^{j}\hat{E}(\mathfrak{m}_{m})^{-} \right)^{\Gal(F_{m, \fr{p}}/F_{n, \fr{p}})}$.
  By $\hat{E}(\mathfrak{m}_{m})^{-}/p^{j}\hat{E}(\mathfrak{m}_{m})^{-} \subset M_{m}^{-}/p^{j}M_{m}^{-}$, we have $\langle \mathbb{H}_{m, \fr{p}}^{-}[p^{j}], y \rangle_{m}=0$.
  Since the map 
  \begin{align}
    H^{1}(F_{n, \fr{p}}, A[p^{j}]) \rarrow H^{1}(F_{m, \fr{p}}, A[p^{j}])^{\Gal(F_{m, \fr{p}}/F_{n, \fr{p}})}
  \end{align}
  is isomorphism by Proposition \ref{p-torはKinfの外}, the element $y$ lies in $H^{1}(F_{n, p}, A[p^{j}])$.
  By Lemma \ref{lem:cor}, for $x \in \mathbb{H}_{n, \fr{p}}^{-}[p^{j}]$, there exist $x' \in \mathbb{H}_{m, \fr{p}}^{-}[p^{j}]$ such that $x = \Cor_{n}^{m}(x')$.
  Thus, we see that 
  \begin{align}
    \langle x,y \rangle_{n} =\langle \Cor_{n}^{m}x', y \rangle_{n}=\langle x', \Res_{n}^{m}y \rangle_{m}=0,
  \end{align}
  that is, $y \in M_{n}^{-}/p^{j}M_{n}^{-}$.

  Finally, it suffices to show that the order of $M_{n}^{-}/p^{j}M_{n}^{-}$ is same to $\mathbb{H}_{n, \fr{p}}^{-}[p^{j}]$.
  Since $M_{n}^{-}/p^{j}M_{n}^{-}$ is the exact annihilator of $\mathbb{H}_{n, \fr{p}}^{-}[p^{j}]$, the pairing
  \begin{align}
    \langle ~, ~ \rangle_{n}: \mathbb{H}_{n, \fr{p}}^{-}[p^{j}] \times \frac{H^{1}(F_{n, \fr{p}}, A[p^{j}])}{M_{n}^{-}/p^{j}M_{n}^{-}} \rarrow \Z/p^{j}\Z
  \end{align}
  is a perfect pairing.
  Thus, we see that $\#\mathbb{H}_{n, \fr{p}}^{-}[p^{j}]=\#\left(H^{1}(F_{n, \fr{p}}, A[p^{j}]) \middle/(M_{n}^{-}/p^{j}M_{n}^{-}) \right)$.
  The isomorphisms \eqref{プラマイのdual} imply $\mathbb{H}_{n,\fr{p}}^{-}[p^{j}] \simeq (\Z/p^{j}\Z)^{2p^{n}}$.
  Therefore, we see $\#\mathbb{H}_{n, \fr{p}}^{-}[p^{j}]=p^{2jp^{n}}$.
  On the other hand, since we have $\# H^{1}(F_{n, \fr{p}}, A[p^{j}])=p^{4jp^{n}}$ by Tate's Euler characteristic formula, we see that $\#(M_{n}^{-}/p^{j}M_{n}^{-})=p^{2jp^{n}}$.
  This concludes the proof of the proposition.
\end{proof}

The analogous result to twist for Proposition \ref{twistなしのexact annihilator} also holds.

\begin{prp}\label{twistありのexact annihilator}
  For any positive integers $j$ and $n$, $\mathbb{H}_{n, \fr{p}}^{s, \pm}[p^{j}]$ is the exact annihilator of $\bb{H}_{n, \fr{p}}^{-s, \pm}[p^{j}]$ with respect to the Tate pairing
  \begin{align}
    H^{1}(F_{n, \fr{p}}, A_{s}[p^{j}]) \times H^{1}(F_{n, \fr{p}}, A_{-s}[p^{j}]) \rarrow \Z/p^{j}\Z.
  \end{align}
\end{prp}
\begin{proof}
  The way of proof is the same as \cite[Proposition 3.4]{Kim13}.
  For any positive integers $j$ and $n$, take enough large $N (>n)$ satisfying $\kappa(\Gal(F_{\infty}/F_{N})) \equiv 1 \mod p^{j}$.
  Then, $A_{s}[p^{j}]$ and $A_{-s}[p^{j}]$ are equal to $A[p^{j}]$ as $G_{F_{N}}$-module.
  By Proposition \ref{twistなしのexact annihilator}, the exact annihilator of $\mathbb{H}_{N, \fr{p}}^{\pm}[p^{j}]$ with respect to the Tate pairing 
  \begin{align}
    \langle ~ , ~ \rangle_{N}:H^{1}(F_{N, \fr{p}}, A[p^{j}]) \times H^{1}(F_{N, \fr{p}}, A[p^{j}]) \rarrow \Z/p^{j}\Z
  \end{align}
  is itself.
  Thus, $\mathbb{H}_{N, \fr{p}}^{s, \pm}[p^{j}]$ is the exact annihilator of $\mathbb{H}_{N, \fr{p}}^{-s, \pm}[p^{j}]$ for enough large $N$.
  
  Since we have $(\mathbb{H}_{m, \fr{p}}^{s,\pm})^{\vee} \simeq \Lambda_{m}^{2}$, we see that $\Cor_{n}^{N}(\mathbb{H}_{N, \fr{p}}^{-s, \pm}[p^{j}])=\mathbb{H}_{n, \fr{p}}^{-s,\pm}[p^{j}]$ by Lemma \ref{lem:cor}.
  Therefore, for any $y \in \mathbb{H}_{n, \fr{p}}^{-s,\pm}[p^{j}]$, there exists $z \in \mathbb{H}_{N,\fr{p}}^{-s, \pm}[p^{j}]$ such that $\Cor_{n}^{N}(z)=y$.
  Then, for any $x \in \mathbb{H}_{n, \fr{p}}^{s,\pm}[p^{j}]$, we see that
  \begin{align}
    \langle x,y \rangle_{n}=\langle x, \Cor_{n}^{N}(z) \rangle_{n}=\langle \Res_{n}^{N}(x), z \rangle_{N}=0.
  \end{align}
  By Tate's Euler characteristic formula, we can calculate $\# H^{1}(F_{n, \fr{p}}, A_{-s}[p^{j}])=p^{4jp^{n}}$.
  As the proof of Proposition \ref{twistなしのexact annihilator}, $\mathbb{H}_{n,\fr{p}}^{s,\pm}[p^{j}]$ is the exact annihilator of $\mathbb{H}_{n, \fr{p}}^{-s, \pm}[p^{j}]$ by comparing the orders.
\end{proof}

\section{Control theorem}
\noindent
We continue to use the same notations as in the previous sections and fix $\fr{p}$ and $w$.
By Proposition \ref{p-torはKinfの外}, the map $f_{n}^{\pm}:\mathbb{H}_{n,\fr{p}}^{\pm}[\omega_{n}^{\pm}] \rarrow \mathbb{H}_{w}^{\pm}[\omega_{n}^{\pm}]$ is injective.
For this map $f_{n}^{\pm}$, we also obtain the following:

\begin{lem}\label{コントロール定理の補題}
  For any positive integers $n$, the order of $\Coker f_{n}^{\pm}$ is finite.
\end{lem}
\begin{proof}
  This proof is the same way of \cite[Lemma 5.1]{Agb-How05}.
  It is sufficient to prove a claim that the $\Lambda$-corank of $\mathbb{H}_{n, \fr{p}}^{\pm}[\omega_{n}^{\pm}]$ coincides with the one of $\mathbb{H}_{w}^{\pm}[\omega_{n}^{\pm}]$.
  Since we have $(\mathbb{H}_{w}^{\pm})^{\vee} \simeq \Lambda^{2}$ by Proposition \ref{Key prop}, we see that 
  \begin{align}
   \mathbb{H}_{w}^{\pm}[\omega_{n}^{\pm}]^{\vee} &\simeq (\mathbb{H}_{w}^{\pm})^{\vee}/\omega_{n}^{\pm}(\mathbb{H}_{w}^{\pm})^{\vee} \\
   & \simeq (\Lambda/\omega_{n}^{\pm})^{2}.
  \end{align}
  On the other hand, the Pontryagin dual of $\mathbb{H}_{n, \fr{p}}^{\pm}[\omega_{n}^{\pm}]$ is equal to $(\Z_{p}[X]/(\omega_{n}^{\pm}))^{2}$.
  Therefore, it turns out the claim.
\end{proof}

\begin{thm}\label{コントロール定理}
Let $X_{n}^{\pm}$ and $X_{\infty}^{\pm}$ denote the dual of $\Sel_{p}^{\pm}(E/F_{n})$ and $\Sel_{p}^{\pm}(E/F_{\infty})^{\vee}$ respectively. 
Then, the kernel and the cokernel of a canonical homomorphism
  \begin{align}
    X_{\infty}^{\pm}/\omega_{n}^{\pm}X_{\infty}^{\pm} \rarrow X_{n}^{\pm}/\omega_{n}^{\pm}X_{n}^{\pm}
  \end{align}
  is a finite group.
\end{thm}
 \begin{proof}
  See \cite[Theorem 5.2]{Agb-How05}.
 \end{proof}

\section{Proof of Main Theorem}
\noindent
In this section, we prove the main theorem and the corollary.
The way of the proof is the same as \cite[Section 4]{Gre99} or \cite[Section 3]{Kim13}.
Fix $s \in \Z_{p}$.

\begin{dfn}
  For any non-negative integers $n$ and any primes $v_{n}$ of $F_{n}$, we define the local conditions for $A_{s}$ as
  \begin{align}
    H_{\mathscr{F}^{\pm}}^{1}(F_{n,v_{n}}, A_{s}):=\left\{ \begin{array}{ll}
    \mathbb{H}_{n, \fr{p}}^{s, \pm} & (v_{n} \mid \fr{p})\\
    0 & (v_{n} \nmid \fr{p})
    \end{array}\right. .
  \end{align}
The local conditions $H_{\mathscr{F}^{\pm}}^{1}(F_{n,v_{n}}, T_{-s})$ for $T_{-s}$ is defined by the set of the annihilators of $H_{\mathscr{F}^{\pm}}^{1}(F_{n,v_{n}}, A_{s})$ with respect to the local Tate pairing
\begin{align}
H^{1}(F_{n,v_{n}}, A_{s}) \times H^{1}(F_{n,v_{n}}, T_{-s}) \rarrow \Q_{p}/\Z_{p}.
\end{align}
\end{dfn}

Let $U_{-s}^{\pm}$ be a $\Q_{p}$-subspace of $H^{1}(F_{n,v_{n}}, V_{-s})$ generated by $H_{\mathscr{F}^{\pm}}^{1}(F_{n,v_{n}}, T_{-s})$.

\begin{dfn}
  We define the local condition $H_{\mathscr{F}^{\pm}}^{1}(F_{n,v_{n}}, A_{-s})$ for $A_{-s}$ by the image of $U_{-s}^{\pm}$ under the canonical map $H^{1}(F_{n,v_{n}}, V_{-s}) \rarrow H^{1}(F_{n,v_{n}}, A_{-s})$.
\end{dfn}

\begin{rem}
  The local conditions for $A_{s}$ and $A_{-s}$ are $p$-divisible by the definition of $\bb{H}_{n, \fr{p}}^{s, \pm}$.
  If a prime $v_{n}$ of $F_{n}$ is not lying $p$, then we also have $H_{\mathscr{F}^{\pm}}^{1}(F_{n, v_{n}}, A_{-s})=0$ because $H^{1}(F_{n, v_{n}}, A_{-s})$ is a finite group.
\end{rem}

Note that we defined $\Sigma$ as the set of primes of $F$ including bad all primes of $E$, primes lying over $p$, and infinity primes.
We define the following:
 \begin{align}
  &P_{n}:=\prod_{v_{n} \mid \fr{q}, \fr{q} \in \Sigma} H^{1}(F_{n, v_{n}}, A_{s}), \quad 
  L_{n}^{\pm}:=\prod_{v_{n} \mid \fr{q}, \fr{q} \in \Sigma} H_{\mathscr{F}^{\pm}}^{1} (F_{n,v_{n}}, A_{s}).
\end{align}
In the case of $n=0$, we omit the indexes and denote $P$ and $L^{\pm}$ respectively.

\begin{dfn}
  Let $\loc_{n}: H^{1}(F_{\Sigma}/F_{n}, A_{s}) \rarrow P_{n}$ be the homomorphisms induced by the global-local map, and $\psi_{n}:H^{1}(F_{\Sigma}/F_{n}, A_{s}) \rarrow P_{n}/L_{n}^{\pm}$ be the composition map of $\loc_{n}$ and the natural projection $P_{n} \rarrow P_{n}/L_{n}^{\pm}$.
  Then, we define 
  \begin{align}
    S_{A_{s}}^{\pm}(F_{n}):=\Ker \psi_{n}.
  \end{align}
\end{dfn}

 \begin{prp}\label{Poitou-Tate}
  Suppose that $S_{A_{s}}^{\pm}(F_{n})$ is a finite group.
  Then, the map $\psi_{n}:H^{1}(F_{\Sigma}/F_{n}, A_{s}) \rarrow P_{n}/L_{n}^{\pm}$ is surjective.
 \end{prp}
 \begin{proof}
  See \cite[Proposition 3.8]{Kim13} and \cite{Gre99}.
 \end{proof}

\begin{lem}\label{コントロール}
  The kernel and cokernel of the map $S_{A_{s}}^{\pm}(F_{n}) \rarrow S_{A_{s}}^{\pm}(F_{\infty})^{\Gal(F_{\infty}/F_{n})}$ are finite group and bounded as $n$ varies.
\end{lem}
\begin{proof}
  See \cite[Proposition 3.9]{Kim13}.
\end{proof}

\begin{prp}\label{目標の全射}
  If the plus/minus Selmer group $\Sel_{p}^{\pm}(E/F_{\infty})$ is $\Lambda$-cotorsion, then
  \begin{align}
    H^{1}(F_{\Sigma}/F_{\infty}, A) \rarrow \prod_{\substack{v_{\infty} \mid \fr{q},~ \fr{q} \in \Sigma}} \frac{H^{1}(F_{\infty, v_{\infty}}, A)}{H_{\mathscr{F}^{\pm}}^{1}(F_{\infty, v_{\infty}}, A)} 
  \end{align}
  is surjective.
\end{prp}
\begin{proof}
  See \cite[Proposition 3.10]{Kim13}.
\end{proof}

By Proposition \ref{目標の全射} and the definition of the plus/minus Selmer groups, if let
\begin{align}
  \cal{P}_{E}^{\Sigma}(F_{\infty})^{\pm} := \prod_{\substack{v_{\infty} \mid \fr{q},~ \fr{q} \in \Sigma}} \frac{H^{1}(F_{\infty, v_{\infty}}, A)}{H_{\mathscr{F}^{\pm}}^{1}(F_{\infty, v_{\infty}}, A)},
\end{align}
then we obtain the following exact sequence as $\Lambda$-module: 
\begin{align}
  0 \rarrow \Sel_{p}^{\pm}(E/F_{\infty}) \rarrow H^{1}(F_{\Sigma}/F_{\infty}, A) \rarrow \mathcal{P}_{E}^{\Sigma}(F_{\infty})^{\pm} \rarrow 0. \label{important exact}
\end{align}

\begin{prp}\label{1次と2次のLambda-corank}
  If $\Sel_{p}^{\pm}(E/F_{\infty})$ is a $\Lambda$-cotorsion, then the $\Lambda$-corank of $H^{1}(F_{\Sigma}/F_{\infty}, A)$ is equal to $[F:\Q]$, and $H^{2}(F_{\Sigma}/F_{\infty}, A)$ is a $\Lambda$-cotorsion.
\end{prp}
\begin{proof}
  Fix any $\fr{q} \in \Sigma$.
  In case of $\fr{q} \neq \fr{p}$, we see that $\prod_{v_{\infty} \mid \fr{q}} H^{1}(F_{\infty, v_{\infty}}, A)$ is a $\Lambda$-cotorsion by the similar way of \cite{Gre99}.
  In case of $\fr{q} \mid \fr{p}$, the $\Lambda$-corank of $H^{1}(F_{\infty, w}, A)$ is $2[F:\Q]$ by \cite[Propostion 1]{Gre89} and the $\Lambda$-corank of $\mathbb{H}_{w}^{\pm}$ is $[F_{\fr{p}}:\Q_{p}]=[F:\Q]$ by Proposition \ref{Key prop}.
  Therefore, we see 
  \begin{align}
    \corank_{\Lambda}\mathcal{P}_{E}^{\Sigma}(F_{\infty})^{\pm}=2[F:\Q]-[F:\Q]=[F:\Q].
  \end{align}
  By the exact sequence \eqref{important exact} and the assumption that $\Sel_{p}^{\pm}(E/F_{\infty})$ is a $\Lambda$-cotorsion, we obtain 
  \begin{align}
    \corank_{\Lambda} H^{1}(F_{\Sigma}/F_{\infty}, A)=\corank_{\Lambda} \mathcal{P}_{E}^{\Sigma}(F_{\infty})^{\pm}=[F:\Q].
  \end{align}
  By \cite[Proposition 3]{Gre89}, we also obtain $\corank_{\Lambda} H^{2}(F_{\Sigma}/F_{\infty}, A)=0$.
\end{proof}

\begin{prp}\label{Greenbergからの引用}
  If $H^{2}(F_{\Sigma}/F_{\infty}, A)$ is a $\Lambda$-cotorsion, then $H^{1}(F_{\Sigma}/F_{\infty}, A)$ has no non-trivial $\Lambda$-submodules with finite index.
\end{prp}
\begin{proof}
  The proof is the same wey of \cite[Proposition 3.13]{Kim13} or \cite[Proposition 4.9]{Gre99}.
\end{proof}

From the above preparations, we can prove the main theorem.

\begin{thm}\label{主定理}
  Assume that the formal group $\hat{E}$ of $E$ defined over $\cal{O}_{\fr{p}}$ is isomorphic to the Lubin-Tate formal group of height 2 with a parameter $-p$.
  Then, if $\Sel_{p}^{\pm}(E/F_{\infty})$ is a $\Lambda$-cotorsion, $\Sel_{p}^{\pm}(E/F_{\infty})$ has no non-trivial subgroups with finite index.
\end{thm}
\begin{proof}
  Choose $s \in \Z_{p}$ such that $S_{A_{s}}^{\pm}(F_{\infty})^{\Gal(F_{\infty}/F_{n})}$ is a finite group.
  Then, the map 
  \begin{align}
    H^{1}(F_{\Sigma}/F, A_{s}) \rarrow P/L^{\pm}  
  \end{align}
  is surjective by Lemma \ref{Poitou-Tate}.
  By the Hochschild-Serre spectrum sequence, the map
  \begin{align}
    H^{1}(F_{\Sigma}/F, A_{s}) \rarrow H^{1}(F_{\Sigma}/F_{\infty}, A_{s})^{\Gamma} \label{コントロールっぽい式}
  \end{align}
  is surjective.
  In the case of $\fr{q} \nmid p$, we see that $H^{1}(F_{\fr{q}}, A_{s}) \rarrow \left( \prod_{v_{\infty} \mid \fr{q}} H^{1}(F_{\infty, v_{\infty}}, A_{s}) \right)^{\Gamma}$ is surjective by the similar way of the above argument.
  In the case of $\fr{p} \mid p$, let $w$ be the prime of $F_{\infty}$ lying over $w$.
  Since we have $( \mathbb{H}_{w}^{s, \pm} )^{\vee} \simeq \Lambda_{\mathcal{O}_{p}}$ by Proposition \ref{Key prop}, we see $( \mathbb{H}_{w}^{s, \pm} )_{\Gamma} \simeq ((\mathbb{H}_{w}^{s, \pm})^{\vee})^{\Gamma} \simeq (\Lambda_{\mathcal{O}_{p}})^{\Gamma}=0$.
  Therefore, we obtain the following exact sequence
  \begin{align}
    0 \rarrow (\mathbb{H}_{w}^{s, \pm} )^{\Gamma} \rarrow H^{1}(F_{\infty, w}, A_{s})^{\Gamma} \rarrow \left( H^{1}(F_{\infty, w}, A_{s})/\mathbb{H}_{w}^{s, \pm} \right)^{\Gamma} \rarrow 0.
  \end{align}
  To summarize the above arguments, the map $P/L^{\pm} \rarrow (P_{\infty}/L_{\infty}^{\pm})^{\Gamma}$ is surjective.

  Since the maps $H^{1}(F_{\Sigma}/F, A_{s}) \rarrow P/L^{\pm}$ and \eqref{コントロールっぽい式} are surjective, the map $H^{1}(F_{\Sigma}/F_{\infty}, A_{s})^{\Gamma} \rarrow (P_{\infty}/L_{\infty}^{\pm})^{\Gamma}$ is also surjective.
  By Proposition \ref{目標の全射}, we see that 
  \begin{align}
    0 \rarrow S_{A_{s}}^{\pm}(F_{\infty}) \rarrow H^{1}(F_{\Sigma}/F_{\infty}, A_{s}) \rarrow P_{\infty}/L_{\infty}^{\pm} \rarrow 0
  \end{align}
  is an exact sequence.
  This induces the following exact sequence:
  \begin{align}
    H^{1}(F_{\Sigma}/F_{\infty}, A_{s})^{\Gamma} \rarrow (P_{\infty}/L_{\infty}^{\pm})^{\Gamma} \rarrow S_{A_{s}}^{\pm}(F_{\infty})_{\Gamma} \rarrow H^{1}(F_{\Sigma}/F_{\infty}, A_{s})_{\Gamma}.
  \end{align}
  By Proposition \ref{Greenbergからの引用}, we see $H^{1}(F_{\Sigma}/F_{\infty}, A_{s})_{\Gamma}=0$.
  Since $H^{1}(F_{\Sigma}/F_{\infty}, A_{s})^{\Gamma} \rarrow (P_{\infty}/L_{\infty}^{\pm})^{\Gamma}$ is surjective, we obtain $S_{A_{s}}^{\pm}(F_{\infty})_{\Gamma}=\Sel_{p}^{\pm}(E/F_{\infty})_{\Gamma} \otimes \kappa^{s}=0$, that is $\Sel_{p}^{\pm}(E/F_{\infty})_{\Gamma}=0$.
  This means that $\Sel_{p}^{\pm}(E/F_{\infty})$ has no non-trivial $\Lambda$-submodule with finite index.
\end{proof}

We want to give heavy thanks to Takenori Kataoka for improving the following corollary.

\begin{cor}
  Suppose that $\Sel_{p}(E/F)$ is a finite group.
  Then, $\Sel_{p}^{\pm}(E/F_{\infty})$ is a $\Lambda$-cotorsion and if let $(f^{\pm}) \subset \Lambda$ denote the characteristic ideals of $\Sel_{p}^{\pm}(E/F_{\infty})^{\vee}$, then we have
  \begin{align}
    |f^{\pm}(0)| \sim \# \Sel_{p}(E/F) \cdot \prod_{v} c_{v},
  \end{align}
  where $c_{v}$ is the Tamagawa number at a prime $v$ of $F$.
\end{cor}

\begin{proof}
  Let $\cal{P}_{E}^{\Sigma}(F):=\prod_{\fr{q} \in \Sigma}H^{1}(F_{\fr{q}}, A)/(E(F_{\fr{q}})\otimes (\Q_{p}/\Z_{p}))$.
  Since $\Sel_{p}(E/F)$ is a finite group, we have the following commutative diagram by Proposition \ref{Poitou-Tate}:
  \begin{align}
    \xymatrix{
    0 \ar[r] & \Sel_{p}(E/F) \ar[r] \ar[d]^-{a} & H^{1}(F_{\Sigma}/F, A) \ar[r] \ar[d] & \cal{P}_{E}^{\Sigma}(F) \ar[r] \ar[d]^-{\prod g_{\fr{q}}} & 0 \\
    0 \ar[r] & \Sel_{p}^{\pm}(E/F_{\infty})^{\Gamma} \ar[r] & H^{1}(F_{\Sigma}/F_{\infty}, A)^{\Gamma} \ar[r] & (\cal{P}_{E}^{\Sigma}(F_{\infty})^{\pm})^{\Gamma}
    }.
  \end{align}
  By the inflation-restriction sequence and Proposition \ref{p-torはKinfの外}, the middle vertex map is an isomorphism.
  Thus, we see that the left vertex map $a$ is injective and the isomorphism
  \begin{align}
    \Ker\left( \prod_{\fr{q} \in \Sigma} g_{\fr{q}} \right) \simeq \Coker(a) \label{eq:Coker}
  \end{align}
  by the snake lemma.
  In case of $\fr{q} \in \Sigma$ not dividing $p$, we see that $\#\Ker g_{\fr{q}} = c_{\fr{q}}^{(p)}$ by \cite[Lemma 4.4]{Gre99}, where $c_{\fr{q}}^{(p)}$ is the highest power of $p$ dividing $c_{\fr{q}}$.
  In case of $\fr{p} \mid p$, $g_{\fr{p}}$ is injective since we have $E(F_{\fr{p}}) \otimes (\Q_{p}/\Z_{p})=(\mathbb{H}_{w}^{\pm})^{\Gamma}$. 
  From the above, we see that $\Ker\left( \prod_{\fr{q}} g_{\fr{q}} \right)$ is a finite group and the order is equal to $\prod_{v} c_{v}$ up to $\Z_{p}$-units.
  Hence, the exact sequence 
  \begin{align}
    \xymatrix{
      0 \ar[r] & \Sel_{p}(E/F) \ar[r]^-{a} & \Sel_{p}^{\pm}(E/F_{\infty})^{\Gamma} \ar[r] & \Coker (a) \ar[r] & 0
    }
  \end{align}
  imply that $\Sel_{p}^{\pm}(E/F_{\infty})^{\Gamma}$ is also a finite group.
  Therefore, $\Sel_{p}^{\pm}(E/F_{\infty})$ is a $\Lambda$-cotorsion.
  On the other hand, we see that
  \begin{align}
    \# \Ker\left( \prod_{\fr{q} \in \Sigma} g_{\fr{q}} \right) = \#\Coker\left( \Sel_{p}(E/F) \rarrow \Sel_{p}^{\pm}(E/F_{\infty})^{\Gamma} \right) =\frac{\# \Sel_{p}^{\pm}(E/F_{\infty})^{\Gamma}}{\#\Sel_{p}(E/F)}.
  \end{align}
  From the above, we obtain
  \begin{align}
    \Sel_{p}^{\pm}(E/F_{\infty})^{\Gamma}=\# \Sel_{p}(E/F) \cdot \#\Ker \left( \prod_{\fr{q} \in \Sigma} g_{\fr{q}} \right) \sim \Sel_{p}(E/F) \cdot \prod_{v} c_{v}.
  \end{align}
  Since we have $\Sel_{p}^{\pm}(E/F_{\infty})_{\Gamma}=0$ by Theorem \ref{主定理}, we obtain 
  \begin{align}
    |f^{\pm}(0)| \sim \frac{\#\Sel_{p}^{\pm}(E/F_{\infty})^{\Gamma}}{\#\Sel_{p}^{\pm}(E/F_{\infty})_{\Gamma}} = \#\Sel_{p}^{\pm}(E/F_{\infty})^{\Gamma}
  \end{align}
  by \cite[Lemma 4.2]{Gre99}.
\end{proof}

\bibliographystyle{plain}
\bibliography{cite}
\end{document}